\documentclass[11pt]{amsart}
\usepackage{amsmath}
\usepackage{amssymb}
\usepackage{amsthm}
\usepackage{latexsym}
\usepackage{graphicx}
\usepackage{hyperref}
\usepackage{enumerate}
\usepackage[all]{xy}

\newtheorem{thm}{Theorem}

\newtheorem{lem}[thm]{Lemma}
\newtheorem{prop}[thm]{Proposition}

\newtheorem{thmintro}{Theorem}

\newcommand{\Z}{\mathbb Z}
\newcommand{\Q}{\mathbb Q}

\newcommand{\mc}{\mathcal}

\newcommand{\mr}{\mathrm}
\newcommand{\enuma}[1]{\begin{enumerate}[\textup{(}a\textup{)}] {#1} \end{enumerate}}

\newcommand{\red}{\mathrm{red}}

\begin{document}

\title[Conjugacy of Levi subgroups]{Conjugacy of Levi subgroups of reductive groups\\
and a generalization to linear algebraic groups}
\author{Maarten Solleveld}
\address{IMAPP, Radboud Universiteit \\
Heyendaalseweg 135, 6525AJ Nijmegen, the Netherlands}
\email{m.solleveld@science.ru.nl}
\date{\today}
\thanks{The author is supported by a NWO Vidi grant "A Hecke algebra approach to the 
local Langlands correspondence" (nr. 639.032.528).}
\maketitle

\begin{abstract}
We investigate Levi subgroups of a connected reductive algebraic group $\mc G$,
over a ground field $K$. We parametrize their conjugacy classes in terms of sets of
simple roots and we prove that two Levi $K$-subgroups of $\mc G$ are rationally 
conjugate if and only if they are geometrically conjugate.

These results are generalized to arbitrary connected linear algebraic $K$-groups. 
In that setting the appropriate analogue of a Levi subgroup is derived from the
notion of a pseudo-parabolic subgroup.
\end{abstract}
\vspace{1cm}

\section{Introduction}

Let $\mc G$ be a connected reductive group over a field $K$. It is well-known that 
conjugacy classes of parabolic $K$-subgroups correspond bijectively to set of simple
roots (relative to $K$). Further, two parabolic $K$-subgroups are $\mc G (K)$-conjugate
if and only if they are conjugate by an element of $\mc G (\overline K)$. In other
words, rational and geometric conjugacy classes coincide.

By a Levi $K$-subgroup of $\mc G$ we mean a Levi factor of some parabolic $K$-subgroup 
of $\mc G$. Such groups play an important role in the representation theory of reductive 
groups, via parabolic induction.
Conjugacy of Levi subgroups, also known as association of parabolic
subgroups, has been studied less. Although their rational conjugacy classes are
known (see \cite[Proposition 1.3.4]{Cas}), it appears that so far these have not
been compared with geometric conjugacy classes.

Let $\Delta_K$ be the set of simple roots for $\mc G$ with respect to a maximal 
$K$-split torus $\mc S$. For every subset $I_K \subset \Delta_K$ there exists a standard
Levi $K$-subgroup $\mc L_{I_K}$. We will prove:

\begin{thmintro}\label{thm:A}
Let $\mc G$ be a connected reductive $K$-group. Every Levi $K$-subgroup of $\mc G$ is 
$\mc G (K)$-conjugate to a standard Levi $K$-subgroup. 

For two standard Levi $K$-subgroups $\mc L_{I_K}$ and $\mc L_{J_K}$
the following are equivalent:
\begin{itemize}
\item $I_K$ and $J_K$ are associate under the Weyl group $W(\mc G,\mc S)$;
\item $\mc L_{I_K}$ and $\mc L_{J_K}$ are $\mc G (K)$-conjugate;
\item $\mc L_{I_K}$ and $\mc L_{J_K}$ are $\mc G (\overline K)$-conjugate.
\end{itemize}
\end{thmintro}

The first claim and the first equivalence are folklore and not hard to show. 
The meat of the theorem is the equivalence of $\mc G (K)$-conjugacy and 
$\mc G (\overline K)$-conjugacy, that is, of rational conjugacy and geometric conjugacy. 
Our proof of that equivalence involves reduction steps and a case-by-case analysis for 
quasi-split absolutely simple groups. It occupies Section \ref{sec:1} of the paper.\\

Our main result is a generalization of Theorem \ref{thm:A} to arbitrary connected linear 
algebraic groups. There we replace the notion of a Levi subgroup by that of a 
\emph{pseudo-Levi subgroup}. By definition, a pseudo-Levi $K$-subgroup of $\mc G$ is the 
intersection of two opposite pseudo-parabolic $K$-subgroups of $\mc G$. We refer to
\cite[\S 2.1]{CGP} and the start of Section \ref{sec:2} for more background.
For reductive groups, pseudo-Levi subgroups are the same as Levi subgroups. When $\mc G$ 
does not admit a Levi decomposition, these pseudo-Levi subgroups are the best analogues. 
In the representation theory of pseudo-reductive groups over local fields (of positive
characteristic), these pseudo-Levi subgroups play a key role \cite[\S 4.1]{Sol}.

We prove that Theorem \ref{thm:A} has a natural analogue in the "pseudo"-setting:

\begin{thmintro}\label{thm:B}
Let $\mc G$ be a connected linear algebraic $K$-group. Every pseudo-Levi $K$-subgroup of 
$\mc G$ is $\mc G (K)$-conjugate to a standard pseudo-Levi $K$-subgroup. 

For two standard pseudo-Levi $K$-subgroups $\mc L_{I_K}$ and $\mc L_{J_K}$
the following are equivalent:
\begin{itemize}
\item $I_K$ and $J_K$ are associate under the Weyl group $W(\mc G,\mc S)$;
\item $\mc L_{I_K}$ and $\mc L_{J_K}$ are $\mc G (K)$-conjugate;
\item $\mc L_{I_K}$ and $\mc L_{J_K}$ are $\mc G (\overline K)$-conjugate.
\end{itemize}
\end{thmintro}

Our arguments rely mainly on the structure theory of linear algebraic groups and 
pseudo-reductive groups developed by Conrad, Gabber and Prasad \cite{CGP,CP}. 
The first claim and the first equivalence are quickly dealt with in Lemma \ref{lem:2.2}. 
Like for reductive groups, the hard part is the equivalence of rational and geometric conjugacy. 
The proof of that constitutes the larger part of Section \ref{sec:2}, from Theorem
\ref{thm:2.7} onwards. We make use of Theorem \ref{thm:A} and of deep classification results
about absolutely pseudo-simple groups \cite{CP}.\\

\textbf{Acknowledgements.}

We thank Jean-Loup Waldspurger for explaining us important steps in the proof of Theorem
\ref{thm:A} and Gopal Prasad for pointing out some subtleties in \cite{CGP}.
\vspace{5mm}

\section{Connected reductive groups}
\label{sec:1}

Let $K$ be field with an algebraic closure $\overline K$ and a separable closure 
$K_s \subset \overline K$. Let $\Gamma_K$ be the Galois group of $K_s / K$.

Let $\mc G$ be a connected reductive $K$-group. Let $\mc T$ be a maximal torus 
of $\mc G$ with character lattice $X^* (\mc T)$. let $\Phi (\mc G, \mc T) \subset X^* (\mc T)$ 
be the associated root system. We also fix a Borel subgroup $\mc B$ of $\mc G$ 
containing $\mc T$, which determines a basis $\Delta$ of $\Phi (\mc G, \mc T)$. For 
every $\gamma \in \Gamma_K$ there exists a $g_\gamma \in \mc G (K_s)$ such that
\[
g_\gamma \gamma (\mc T) g_\gamma^{-1} = \mc T \quad \text{and} \quad
g_\gamma \gamma (\mc B) g_\gamma^{-1} = \mc B . 
\]
One defines the $\mu$-action of $\Gamma_K$ on $\mc T$ by
\begin{equation}\label{eq:1.2}
\mu_{\mc B}(\gamma) (t) = \mr{Ad}(g_\gamma) \circ \gamma (t) .
\end{equation}
This also determines an action $\mu_{\mc B}$ of $\Gamma_K$ on $\Phi (\mc G,\mc T)$,
which stabilizes $\Delta$.

Let $\mc S$ be a maximal $K$-split torus in $\mc G$. By \cite[Theorem 13.3.6.(i)]{Spr} 
applied to $Z_{\mc G}(\mc S)$, we may assume that $\mc T$ is defined over $K$ and 
contains $\mc S$. Then $Z_{\mc G}(\mc S)$ is a minimal $K$-Levi subgroup of $\mc G$. Let
\[
\Delta_0 := \{ \alpha \in \Delta : \mc S \subset \ker \alpha \}
\]
be the set of simple roots of $(Z_{\mc G}(\mc S), \mc T)$. It is known that $\Delta_0$
is stable under $\mu_{\mc B}(\Gamma_K)$ \cite[Proposition 15.5.3.i]{Spr}, so $\mu_{\mc B}$ 
can be regarded as a group homomorphism $\Gamma_K \to \mr{Aut}(\Delta,\Delta_0)$. The 
triple $(\Delta ,\Delta_0, \mu_{\mc B})$ is called the index of $\mc G$ \cite[\S 15.5.5]{Spr}.

Recall from \cite[Lemma 15.3.1]{Spr} that the root system $\Phi (\mc G, \mc S)$ is the 
image of $\Phi (\mc G, \mc T)$ in $X^* (\mc S)$, without 0. The set of simple roots $\Delta_K$ 
of $(\mc G, \mc S)$ can be identified with $(\Delta \setminus \Delta_0 ) / \mu_{\mc B}(\Gamma_K)$.
The Weyl group of $(\mc G, \mc S)$ can be expressed in various ways:
\begin{equation}\label{eq:1.1}
\begin{aligned}
W(\mc G,\mc S) & = N_{\mc G}(\mc S) / Z_{\mc G}(\mc S) \cong
N_{\mc G (K)}(\mc S(K)) / Z_{\mc G (K)}(\mc S (K)) \\
& \cong N_{\mc G}(\mc S,\mc T) / N_{Z_{\mc G}(\mc S)}(\mc T) =
\big( N_{\mc G}(\mc S,\mc T) / \mc T \big) \big/ \big( N_{Z_{\mc G}(\mc S)}(\mc T) / \mc T \big) \\
& \cong \mr{Stab}_{W(\mc G,\mc T)} (\mc S) / W(Z_{\mc G}(\mc S), \mc T) . 
\end{aligned}
\end{equation}
Let $\mc P_{\Delta_0} = Z_{\mc G}(\mc S) \mc B$ the minimal parabolic $K$-subgroup of $\mc G$ 
associated to $\Delta_0$. It is well-known \cite[Theorem 15.4.6]{Spr} 
that the following sets are canonically in bijection:
\begin{itemize}
\item $\mc G (K)$-conjugacy classes of parabolic $K$-subgroups of $\mc G$;
\item standard (i.e. containing $\mc P_{\Delta_0}$) parabolic $K$-subgroups of $\mc G$;
\item subsets of $(\Delta \setminus \Delta_0 ) /  \mu_{\mc B}(\Gamma_K)$;
\item $ \mu_{\mc B}(\Gamma_K)$-stable subsets of $\Delta$ containing $\Delta_0$.
\end{itemize}
Comparing these criteria over $K$ and over $\overline K$, we see that two parabolic $K$-subgroups
of $\mc G$ are $\mc G (K)$-conjugate if and only if they are $\mc G (\overline K)$-conjugate.

By a parabolic pair for $\mc G$ we mean a pair $(\mc P,\mc L)$, where $\mc L \subset \mc P$ is
a parabolic subgroup and $\mc L$ is a Levi factor of $\mc P$. We say that the pair is defined
over $K$ if both $\mc P$ and $\mc L$ are so.
By a Levi subgroup of $\mc G$ we mean a Levi factor of some parabolic subgroup of $\mc G$.
Equivalently, a Levi $K$-subgroup of $\mc G$ is the centralizer of a $K$-split torus in $\mc G$.

With \cite[Lemma 15.4.5]{Spr} every $\mu_{\mc B}(\Gamma_K)$-stable subset $I \subset \Delta$ 
containing $\Delta_0$ gives rise to a standard Levi $K$-subgroup $\mc L_I$ of $\mc G$, 
namely the group generated by $Z_{\mc G}(\mc S)$ and the root subgroups for roots in
$\Z I \cap \Phi (\mc G,\mc T)$. By construction $\mc L_I$ is a Levi factor of the standard
parabolic $K$-subgroup $\mc P_I$ of $\mc G$. In the introduction we denoted $\mc L_I$ by
$\mc L_{I_K}$, where $I_K = (I \setminus \Delta_0) / \mu_{\mc B}(\Gamma_K)$.

Two parabolic $K$-subgroups of $\mc G$ are called associate if their Levi factors are 
$\mc G (K)$-conjugate. As Levi factors are unique up to conjugation (see the proof of Lemma 
\ref{lem:1}.a below), there is a natural bijection between the set of $\mc G (K)$-conjugacy 
classes of Levi $K$-subgroups of $\mc G$ and the set of association classes of parabolic
$K$-subgroups of $\mc G$. The explicit description of these sets is known, for instance from
\cite[Proposition 1.3.4]{Cas}. Unfortunately we could not find a complete proof of these 
statements in the literature, so we provide it here.

\begin{lem}\label{lem:1}
\enuma{
\item Every Levi $K$-subgroup of $\mc G$ is $\mc G (K)$-conjugate to a standard Levi
$K$-subgroup of $\mc G$.
\item For two standard Levi $K$-subgroups $\mc L_I$ and $\mc L_J$ the following are
equivalent:
\begin{enumerate}[(i)]
\item $\mc L_I$ and $\mc L_J$ are $\mc G (K)$-conjugate;
\item $(I \setminus \Delta_0) / \mu_{\mc B}(\Gamma_K)$ and 
$(J \setminus \Delta_0) / \mu_{\mc B}(\Gamma_K)$ are $W (\mc G, \mc S)$-associate.
\end{enumerate} }
\end{lem} 
\begin{proof}
(a) Let $\mc P$ be a parabolic $K$-subgroup of $\mc G$ with a Levi factor $\mc L$
defined over $K$. Since $\mc P$ is $\mc G (K)$-conjugate to a standard parabolic
subgroup $\mc P_I$ \cite[Theorem  15.4.6]{Spr}, $\mc L$ is $\mc G (K)$-conjugate
to a Levi factor of $\mc P_I$. By \cite[Proposition 16.1.1]{Spr} any two such factors
are conjugate by an element of $\mc P_I (K)$. In particular $\mc L$ is 
$\mc G (K)$-conjugate to $\mc L_I$.\\
(b) Suppose that (ii) is fulfilled, that is, 
\[
w(I \setminus \Delta_0) /  \mu_{\mc B}(\Gamma_K) = (J \setminus \Delta_0) /  
\mu_{\mc B}(\Gamma_K) \quad \text{for some } w \in W(\mc G,\mc S) .
\]
Let $\bar w \in N_{\mc G (K)}(\mc S(K))$ be a lift of $w$. Then $\bar w \mc L_I {\bar w}^{-1}$
contains $Z_{\mc G} (\mc S)$ and 
\[
\Phi (\bar w \mc L_I {\bar w}^{-1} ,\mc S) = w \Phi (\mc L_I,\mc S) = \Phi (\mc L_J ,\mc S) .
\]
Hence $\bar w \mc L_I {\bar w}^{-1} = \mc L_J$, showing that (i) holds.

Conversely, suppose that (ii) holds, so $g \mc L_I g^{-1} = \mc L_J$ for some $g \in 
\mc G (K)$. Then $g \mc S g^{-1}$ is a maximal $K$-split torus of $\mc L_J$.
By \cite[Theorem 15.2.6]{Spr} there is a $l \in \mc L_J (K)$ such that 
$l g \mc S g^{-1} l^{-1} = \mc S$. Thus $(lg) \mc L_I (l g)^{-1} = \mc L_J$ and
$lg \in N_{\mc G}(\mc S)$. Let $w_1$ be the image of $l g$ in $W(\mc G,\mc S)$. Then 
$w_1 (\Phi (\mc L_I,\mc S)) = \Phi (\mc L_J,\mc S)$, so $w_1 \big( (I \setminus \Delta_0) 
/ \mu_{\mc B}(\Gamma_K) \big)$ is a basis of $\Phi (\mc L_J,\mc S)$. Any two bases 
of a root system are associate under its Weyl group, so there exists a 
$w_2 \in W(\mc L_J,\mc S) \subset W(\mc G,\mc S)$ such that 
\[
w_2 w_1 \big( (I \setminus \Delta_0) / \mu_{\mc B}(\Gamma_K) \big) = 
(J \setminus \Delta_0) / \mu_{\mc B}(\Gamma_K) . \qedhere
\]
\end{proof}

When $\mc G$ is $K$-split, $\Delta_0$ is empty and the action of $\Gamma_K$ is trivial.
Then Lemma \ref{lem:1} says that $\mc L_I$ and $\mc L_J$ are $\mc G (K)$-conjugate
if and only if $I$ and $J$ are $W(\mc G,\mc T)$-associate. With $\overline K$ instead of
$K$ we would obtain the same criterion. In particular $\mc L_I$ and $\mc L_J$ are
$\mc G (K)$-conjugate if and only if they are $\mc G (\overline K)$-conjugate.\\

We want to prove that rational conjugacy and geometric conjugacy of Levi subgroups
is equivalent. More precisely:

\begin{thm}\label{thm:2}
Let $\mc L, \mc L'$ be two Levi $K$-subgroups of $\mc G$. Then $\mc L$ and $\mc L'$
are $\mc G(K)$-conjugate if and only if they are $\mc G (\overline K)$-conjugate.
\end{thm}

The proof consists of several steps:
\begin{itemize}
\item Reduction from reductive to quasi-split $\mc G$.
\item Reduction from reductive (quasi-split) to absolutely simple (quasi-split) $\mc G$.
\item Proof for absolutely simple, quasi-split groups.
\end{itemize}
The first of these three steps is due to Jean-Loup Waldspurger.

Let $\mc G^*$ be a quasi-split $K$-group with an inner twist $\psi : \mc G \to \mc G^*$.
Thus $\psi$ is an isomorphism of $K_s$-groups and there exists a map
$u : \Gamma_K \to \mc G^* (K_s)$ such that
\begin{equation}\label{eq:1.7}
\psi \circ \gamma \circ \psi^{-1} = \mr{Ad}(u(\gamma)) \circ \gamma^*
\qquad \forall \gamma \in \Gamma_K .
\end{equation}
Here $\gamma^*$ denotes the $\Gamma_K$-action which defines the $K$-structure of $\mc G^*$.
We fix a Borel $K$-subgroup $\mc B^*$ of $\mc G^*$ and a maximal $K$-torus
$\mc T^* \subset \mc B^*$ which is maximally $K$-split. In other words, $(\mc B^*, \mc T^*)$
is a minimal parabolic pair of $\mc G^*$, defined over $K$. 
In $\mc G^*$ we also have the parabolic pair 
\[
(\mc P_{\Delta_0}^* , \mc L_{\Delta_0}^*) := (\psi (\mc P_{\Delta_0}), \psi (\mc L_{\Delta_0})) ,
\]
which is defined over $K_s$.
By the conjugacy of minimal parabolic pairs, there exists a $g_0 \in \mc G^* (K_s)$ such that 
\[
g_0 \psi (\mc P_{\Delta_0}) g_0^{-1} \supset \mc B^* \quad \text{ and } \quad
g_0 \psi (\mc L_{\Delta_0}) g_0^{-1} \supset \mc T^* .
\]
Replacing $\psi$ by Ad$(g_0) \circ \psi$, we may assume that 
$\mc P_{\Delta_0}^* \supset \mc B^*$ and $\mc L_{\Delta_0}^* \supset \mc T^*$.

\begin{lem}\label{lem:6} 
\enuma{
\item The parabolic pair $(\mc P_{\Delta_0}^*,\mc L_{\Delta_0}^*)$ is defined over $K$.
\item $u(\gamma) \in \mc L_{\Delta_0}^* (K_s)$ for all $\gamma \in \Gamma_K$.
\item Let $\mc H$ be a $K_s$-subgroup of $\mc G$ containing $\mc L_{\Delta_0}$.
Then $\mc H$ is defined over $K$ if and only if $\psi (\mc H)$ is defined over $K$.
}
\end{lem}
\begin{proof}
(a) Recall that a $K_s$-subgroup of $\mc G$ is defined over $K$ if and only if it is 
$\Gamma_K$-stable. Applying that to $\mc P_{\Delta_0}$ and $\mc L_{\Delta_0}$, 
we see from \eqref{eq:1.7} that Ad$(u(\gamma)) \circ \gamma^*$ stabilizes 
$(\mc P_{\Delta_0}^*,\mc L_{\Delta_0}^*)$. In other words, Ad$(u(\gamma))$ sends 
$(\gamma^* \mc P_{\Delta_0}^*,\gamma^* \mc L_{\Delta_0}^*)$ to 
$(\mc P_{\Delta_0}^*,\mc L_{\Delta_0}^*)$. By the above setup both 
$(\mc P_{\Delta_0}^*,\mc L_{\Delta_0}^*)$ and $(\gamma^* \mc P_{\Delta_0}^*, 
\gamma^* \mc L_{\Delta_0}^*)$ are standard, that is, contain $(\mc B^*,\mc T^*)$. But two
conjugate standard parabolic pairs of $\mc G^*$ are equal, so $\gamma^*$ stabilizes
$(\mc P_{\Delta_0}^*,\mc L_{\Delta_0}^*)$. Hence this parabolic pair is defined over $K$.\\
(b) Now also Ad$(u(\gamma))$ stabilizes $(\mc P_{\Delta_0}^*,\mc L_{\Delta_0}^*)$. As 
every parabolic subgroup is its own normalizer:
\[
u(\gamma) \in N_{\mc G^* (K_s)}(\mc P_{\Delta_0}^*,\mc L_{\Delta_0}^*) = 
N_{\mc P_{\Delta_0}^* (K_s)}(\mc L_{\Delta_0}^*) = \mc L_{\Delta_0}^* (K_s) . 
\]
(c) By part Ad$(u(\gamma))$ stabilizes $\psi (\mc H)$, for any $\gamma \in \Gamma_K$. 
From \eqref{eq:1.7} we see now that $\gamma$ stabilizes $\mc H$ if and only
if it stabilizes $\psi (\mc H)$.
\end{proof}

We thank Jean-Loup Waldspurger for showing us the proof of the next result.

\begin{lem}\label{lem:7}
Suppose that Theorem \ref{thm:2} holds for all quasi-split $K$-groups. Then
it holds for all reductive $K$-groups $\mc G$. 
\end{lem}
\begin{proof}
By Lemma \ref{lem:1}.a it suffices to consider two standard Levi $K$-subgroups $\mc L_I,
\mc L_J$ of $\mc G$. We assume that they are $\mc G (\overline K)$-conjugate. By Lemma
\ref{lem:1}.b this depends only the Weyl group of $(\mc G,\mc T)$, so we can pick
$w \in N_{\mc G (K_s)}(\mc T)$ with $w \mc L_I w^{-1} = \mc L_J$. We denote the images
of these objects (and of $\mc P_I,\mc P_J$) under $\psi$ by a *, e.g. $\mc L^*_I =
\psi (\mc L_I)$. Then $w^* \mc L_I^* {w^*}^{-1} = \mc L_J^*$ and by Lemma \ref{lem:6}.c 
the parabolic pairs $(\mc P_I^*,\mc L_I^*)$ and $(\mc P_J^*,\mc L_J^*)$ are defined over $K$.

Using the hypothesis of the lemma for $\mc G^*$, we pick a $h^* \in \mc G^* (K)$ with 
$h^* \mc L_I^* {h^*}^{-1} = \mc L_J^*$. Write $\mc P^* = h^* \mc P_I^* {h^*}^{-1}$,
$h = \psi^{-1}(h^*)$ and $\mc P := \psi^{-1}(\mc P^*)$. 
Here $\mc P^*$ is defined over $K$ because $\mc P_I^*$ and $h^*$ are. Furthermore 
\[
\mc P^* \supset \mc L_J^* \supset \mc L_{\Delta_0}^* \quad \text{and} \quad
\mc P \supset \mc L_J \supset \mc L_{\Delta_0} ,
\]
so by Lemma \ref{lem:6}.c $\mc P$ is defined over $K$.

Thus the parabolic $K$-subgroups $\mc P_I$ and $\mc P$ of $\mc G$ are conjugate by
$h \in \mc G (K_s)$. Hence they are also $\mc G (K)$-conjugate, say
$g \mc P g^{-1} = \mc P_I$ with $g \in \mc G (K)$. Now $g \mc L_J g^{-1}$ is a Levi
factor of $\mc P_I$ defined over $K$. By \cite[Proposition 16.1.1]{Spr} $g \mc L_J g^{-1}$
is $\mc P_I (K)$-conjugate to $\mc L_I$, so $\mc L_I$ and $\mc L_J$ are $\mc G(K)$-conjugate. 
\end{proof}

\begin{lem}\label{lem:3}
Suppose that Theorem \ref{thm:2} holds for all absolutely simple $K$-groups. Then
it holds for all reductive $K$-groups $\mc G$.

Similarly, if Theorem \ref{thm:2} holds for all absolutely simple, quasi-split $K$-groups,
then it holds for all quasi-split reductive $K$-groups $\mc G$.
\end{lem}
\begin{proof}
The set of standard Levi $K$-subgroups of $\mc G$ does not change when we divide out any 
central $K$-subgroup $\mc Z$ of $\mc G$. In Lemma \ref{lem:1} the criterion (ii) also does not 
change if we divide out $\mc Z$, because $W(\mc G / \mc Z, \mc S / \mc Z) \cong W(\mc G,\mc S)$.
Therefore we may assume that $\mc G$ is of adjoint type. 

Now $\mc G$ is a direct product of $K$-simple groups of adjoint type. If Theorem \ref{thm:2} 
holds for $\mc G'$ and $\mc G''$, then it clearly holds for $\mc G' \times \mc G''$. Thus we
may further assume that $\mc G$ is $K$-simple and of adjoint type.

Then there are simple adjoint $K_s$-groups $\mc G_i$ such that
\begin{equation}\label{eq:1.4}
\mc G \cong \mc G_1 \times \cdots \times \mc G_d \qquad \text{as } K_s \text{-groups.}  
\end{equation}
Since $\mc G$ is $K$-simple, the action of $\Gamma_K$ (which defines the $K$-structure)
permutes the $\mc G_i$ transitively. Write $\mc T_i = \mc T \cap \mc G_i$, so that
$\mc T = \mc T_1 \times \cdots \times \mc T_d$ and
\begin{align}\label{eq:1.3} 
W(\mc G, \mc T) = W(\mc G_1,\mc T_1) \times \cdots \times W(\mc G_d, \mc T_d) , \\
\Phi (\mc G, \mc T) = \Phi (\mc G_1,\mc T_1) \sqcup \cdots \sqcup \Phi (\mc G_d, \mc T_d) .
\end{align}
Put $\Delta^i = \Delta \cap \Phi (\mc G_i,\mc T_i)$ and $\Delta_0^i = 
\Delta_0 \cap \Phi (\mc G_i,\mc T_i)$. Let $\Gamma_i$ be the 
$\Gamma_K$-stabilizer of $\mc G_i$. By \cite[Proposition 15.5.3]{Spr} $\mu_{\mc B}(\Gamma_i)$ 
stabilizes $\Delta_0^i$ and $\mu_{\mc B}(\Gamma) \Delta_0^i = \Delta_0$.

Select $\gamma_i \in \Gamma_K$ with $\gamma_i (\mc G_1) = \mc G_i$ and $\gamma_1 = 1$. Note that
$\mc B_i := \gamma_i (\mc B \cap \mc G_1)$ is a Borel subgroup of $\mc G_i$. To simplify things 
a little bit, we replace $\mc B$ by $\mc B_1 \times \cdots \times \mc B_d$. With this new $\mc B$:
\begin{equation}\label{eq:1.6}
\mu_{\mc B}(\gamma_i) \Delta^1 = \gamma_i (\Delta^1) = \Delta^i \quad \text{and} \quad
\mu_{\mc B}(\gamma_i) \Delta_0^1 = \gamma_i (\Delta_0^1) = \Delta_0^i .
\end{equation}
By Lemma \ref{lem:1}.a it suffices to prove Theorem \ref{thm:2} for standard Levi 
$K$-subgroups $\mc L_I ,\mc L_J$ of $\mc G$, where $\Delta_0 \subset I,J \subset \Delta$ and $I,J$ 
are $\mu_{\mc B}(\Gamma)$-stable. We suppose that $\mc L_I$ and $\mc L_J$ are 
$\mc G (\overline K)$-conjugate, and we have to show that they are also $\mc G (K)$-conjugate.

By \eqref{eq:1.4} the groups $\mc L_I \cap \mc G_i$ and $\mc L_J \cap \mc G_i$ are 
$\mc G_i (K_s)$-conjugate, for $i = 1,\ldots,d$. The absolutely simple group $\mc G_i$ 
is defined over the field $K_i := K_s^{\Gamma_i}$. By the assumption of the current lemma,
$\mc L_I \cap \mc G_i$ and $\mc L_J \cap \mc G_i$ are $\mc G_i (K_i)$-conjugate.

Let $\mc S_i$ be the maximal $K_i$-split torus of $\mc G_i$ such that 
\[
\mc S = \mc S_1 \times \cdots \times \mc S_d \qquad \text{as } K_s \text{-groups.} 
\]
Then $\Gamma_i$ acts trivially on $W(\mc G_i,\mc S_i)$, because the latter is generated 
by $\Gamma_i$-invariant reflections \cite[Lemma 15.3.7.ii]{Spr}. Consider the 
$\mu_{\mc B}(\Gamma_i)$-stable sets $I^i = I \cap \Phi (\mc G_i,\mc T_i)$ and 
$J^i = J \cap \Phi (\mc G_i,\mc T_i)$. By Lemma \ref{lem:1}.b the sets 
$I^i \setminus \Delta_0^i$ and $J^i \setminus \Delta_0^i$ are $W(\mc G_i,\mc S_i)$-associate. 
Pick $w_1 \in W(\mc G_1,\mc S_1)$ with
\[
w_1 ( J^1 \setminus \Delta_0^1 ) = I^1 \setminus \Delta_0^1 .
\]
The analogue of \eqref{eq:1.3} for $\mc S$ reads
\begin{equation}\label{eq:1.5}
W(\mc G,\mc S) = \big( W(\mc G_1,\mc S_1) \times \cdots \times W(\mc G_d, \mc S_d) \big)^{\Gamma_K} .
\end{equation}
Put $w_i = \gamma_i (w_1) \in W(\mc G_i,\mc S_i)$. From \eqref{eq:1.5} we see that 
$w := w_1 \times \cdots \times w_d$ lies in $W(\mc G,\mc S)$. 
By \eqref{eq:1.6} and by the $\mu_{\mc B}(\Gamma_K)$-stability of $I$ and $J$:
\[
w_i ( J^i \setminus \Delta_0^i ) = I^i \setminus \Delta_0^i \qquad \text{for } i = 1,\ldots,d .  
\]
Hence $w (J \setminus \Delta_0 ) = I \setminus \Delta_0$. Now Lemma \ref{lem:1}.b 
says that $\mc L_I$ and $\mc L_J$ are $\mc G (K)$-conjugate.

Finally, we take a closer at the special case where the initial group $\mc G$ was quasi-split
over $K$. Then the group $\mc G_i$ from \eqref{eq:1.4} is quasi-split over $K_i$, for instance
because it admits the $\Gamma_i$-stable Borel subgroup $\mc B_i$. So in the above proof of
Theorem \ref{thm:2} for a quasi-split group $\mc G$, we only need to assume it for the 
quasi-split absolutely simple groups $\mc G_i$.
\end{proof}

When $\mc G$ is quasi-split over $K$, $\Delta_0$ is empty and we can choose $\mc B$ and $\mc T$
defined over $K$, that is, $\Gamma_K$-stable. Then the $\mu$-action of $\Gamma_K$ agrees with 
the action defining the $K$-structure, and it is known from \cite[Proposition 2.4.2]{SiZi} that
\begin{equation}\label{eq:7}
W(\mc G,\mc S) = W(\mc G,\mc T)^{\Gamma_K} . 
\end{equation}
In this case every $\Gamma_K$-stable subset $I$ of $\Delta$ gives rise to standard Levi 
$K$-subgroup $\mc L_I$ of $\mc G$. Lemma \ref{lem:1}.b says that $\mc L_I$ and $\mc L_J$ are
\begin{itemize}
\item $\mc G(\overline K)$-conjugate if and only if $I$ and $J$ are $W(\mc G,\mc T)$-associate;
\item $\mc G(K)$-conjugate if and only if $I$ and $J$ are $W(\mc G,\mc T)^{\Gamma_K}$-associate.
\end{itemize}

\begin{lem}\label{lem:4}
Theorem \ref{thm:2} holds when $\mc G$ is absolutely simple and quasi-split (over $K$).
\end{lem}
\begin{proof}
By Lemma \ref{lem:1} and the remarks after its proof, Theorem \ref{thm:2} holds for 
$K$-split reductive groups. Thus it suffices to consider quasi-split, non-split, absolutely 
simple $K$-groups. In view of Lemma \ref{lem:1}.a, we may assume that $\mc L = \mc L_I$
and $\mc L' = \mc L_J$ are standard Levi $K$-subgroups of $\mc G$. By the above criteria for 
conjugacy, the only things that matter are the root system $\Phi (\mc G,\mc T)$, its Weyl group 
and the Galois action on those. These reductions make a case-by-case consideration feasible. 
In each case, we suppose that $\mc L_I$ and $\mc L_J$ are $\mc G (\overline K)$-conjugate and we 
have to show that $w I = J$ for some $w \in W(\mc G,\mc S) = W(\mc G,\mc T)^{\Gamma_K}$.\\

\textbf{Type $A_n^{(2)}$.} 
The $\Gamma_K$-stable subset $I \subset A_n^{(2)}$ has the form
\[
{A_{n_1}}^2 \times \cdots \times {A_{n_k}}^2 \times A_{n_0}^{(2)} ,
\]
where $n_0$ has the same parity as $n$ and
\[
n_1 + \cdots + n_k + k \leq (n - n_0) / 2 .
\]
Here the connected component $A_{n_0}^{(2)}$ lies in 
the middle of the Dynkin diagram, and all the connected components $A_{n_i}$ occur two times, 
symmetrically around the middle. Similarly $J$ looks like
\[
{A_{m_1}}^2 \times \cdots \times {A_{m_l}}^2 \times A_{m_0}^{(2)} .
\]
Lemma \ref{lem:1}.b tells us that $I$ and $J$ are associate by an element $w$ of $W(\mc G,\mc T) 
\cong S_{n+1}$. Hence the multisets $(n_1,n_1,\ldots,n_k,n_k,n_0)$ and $(m_1,m_1,\ldots,m_l,m_l,m_0)$ 
are equal. Only the element $n_0$ (resp. $m_0$) occurs with odd multiplicity, so $n_0 = m_0$. 
Composing $w$ inside $S_{n+1}$ with a suitable permutation on the components $A_{n_0}$ of
$I$, we may assume that $w$ fixes the subset $A_{m_0}^{(2)} = A_{n_0}^{(2)}$ of $A_n^{(2)}$. 
In ${A_{(n - n_0 - 2)/2}}^2$, the complement of $A_{n_0}^{(2)}$ and the two adjacent simple roots, 
the sets
\[
I' := ( A_{n_1} \times \cdots \times A_{n_k} )^2 \quad \text{and} \quad
J' := ( A_{m_1} \times \cdots \times A_{m_l} )^2 
\]
are associated by $w$. In particular $k = l$. With the group 
$(S_{(n - n_0) / 2}^2)^{\Gamma_K} \cong S_{(n - n_0) / 2}$ 
we can sort $I'$ and $J'$, so that $n_1 \geq \cdots \geq n_k$ and $m_1 \geq \cdots \geq m_k$.
As $I'$ and $J'$ came from the same multiset, they become equal after sorting. This shows that
$w' I' = J'$ for some $w' \in (S_{(n - n_0) / 2}^2)^{\Gamma_K} \subset W(\mc G,\mc T)^{\Gamma_K}$.
In view of \eqref{eq:7}, this says $w' I = J$ with $w' \in W(\mc G,\mc S)$.\\

\textbf{Type $D_n^{(2)}$.}
The $\Gamma_K$-stable subset $I \subset D_n^{(2)}$ has the type 
\[
A_{n_1} \times \cdots \times A_{n_k} \times D_{n_0}^{(2)}
\quad \text{with } n_0 \geq 2 \text{ and }n_1 + \cdots + n_k + k + n_0 \leq n ,
\]
or (when $n_0 = 0)$
\[
A_{n_1} \times \cdots \times A_{n_k} 
\quad \text{with } n_1 + \cdots + n_k + k + 1 \leq n .
\]
Similarly we write 
\[
J = A_{m_1} \times \cdots \times A_{m_l} \times D_{m_0}^{(2)} \quad \text{with } m_0 \neq 1 .
\]
By assumption there exists a $w \in W(D_n)$ such that $w (I) = J$. Suppose that $n_0 \geq 2$
and $w D_{n_0}^{(2)}$ is a component $A_{n_0}$ of $J$. In the standard construction of
the root system $D_n$ in $\Z^n$, the subset $D_{n_0}^{(2)}$ involves precisely $n_0$
coordinates, whereas $A_{n_0}$ involves $n_0 + 1$ coordinates (irrespective of where it is
located in the Dynkin diagram). As $W(D_n) \subset S_n \ltimes \{ \pm 1 \}^n$, applying
$w$ to a set of simple roots does not change the number of involved coordinates. This
contradiction shows that $w$ must map $D_{n_0}^{(2)}$ to $D_{m_0}^{(2)}$ if $n_0 \geq 2$.

For the same reason, if $m_0 \geq 2$, then $w^{-1} D_{m_0}^{(2)}$ must be contained in 
$D_{n_0}^{(2)}$. Hence $n_0 = m_0$ and $w D_{n_0}^{(2)} = D_{m_0}^{(2)}$ whenever 
$n_0 \geq 2$ or $m_0 \geq 2$. Obviously the same conclusion holds in the remaining case 
$n_0 = m_0 = 0$.

Consider the sets of simple roots
\[
I' := A_{n_1} \times \cdots \times A_{n_k} \quad \text{and} \quad
J' := A_{m_1} \times \cdots \times A_{m_l} 
\]
They are associated by $w \in W(D_n)$, so $(n_1,\ldots,n_k) = (m_1,\ldots,m_l)$ as multisets.
Then there exists a $w' \in S_{n - n_0 - 1}$ (or in $S_{n - 2}$ if $n_0 = 0$) with
$w' I' = J'$. Such a $w'$ commutes with the diagram automorphism, so
$w' I = J$ with $w' \in W(D_n)^{\Gamma_K} = W(\mc G,\mc S)$.\\

\textbf{Type $D_4^{(3)}$.}
The cardinality of $I$ is 0, 1, 3 or 4, and for all these sizes there is a unique 
$\Gamma_K$-stable subset of the Dynkin diagram $D_4^{(3)}$. Hence $\mc L_I$ is completely
characterized by its rank $|I| = \mr{rk}(\Phi (\mc L_I,\mc T))$. For each possible rank
there is a unique $\mc G (K)$-conjugacy class of Levi $K$-subgroups, and those Levi subgroups
definitely cannot be $\mc G (\overline K)$-conjugate to Levi subgroups of other ranks.\\

\textbf{Type $E_6^{(2)}$.}
We label the Dynkin diagram as
\[
\begin{array}{c}
\alpha_2 \\ 
\mid \\
\alpha_1 - \alpha_3 - \alpha_4 - \alpha_5 - \alpha_6 
\end{array}
\]
The nontrivial automorphism $\gamma$ exchanges $\alpha_1$ with $\alpha_6$ and $\alpha_3$ with 
$\alpha_5$. Since $\mc L_I$ and $\mc L_J$ are $\mc G (\overline K)$-conjugate, they have 
the same rank $|I| = |J|$. When $|I| = 0$ or $|I| = 6$, this already shows that $J = I$.

For the remaining ranks, we will check that the $W(E_6)$-association classes of $\Gamma_K$-stable
subsets of $E_6$ of that rank are exactly the $W(E_6)^{\Gamma_K}$-association classes. That
suffices, for it implies that the $W(E_6)$-associate sets $I$ and $J$ are already associated 
by an element of $W(E_6)^{\Gamma_K}$.

For $|I| = 1$, the options are $\{\alpha_2\}$ and $\{\alpha_4\}$. These sets 
are associated by an element $w_2 \in \langle s_{\alpha_2},s_{\alpha_4} \rangle \cong S_3$. 
As $\alpha_2$ and $\alpha_4$ are fixed by $\Gamma_K$, $w_2 \in W(E_6)^{\Gamma_K}$.
Hence there is only one $W(E_6)^{\Gamma_K}$-association class of $I$'s of rank 1.

When $|I| = 2$, the possible sets of simple roots are 
\[
I_{2,1} = \{\alpha_2,\alpha_4\} ,\; I_{2,2} = \{ \alpha_3, \alpha_5\} \,\; 
I_{2,3} = \{ \alpha_1, \alpha_6\}.
\]
Among these $I_{2,1} \cong A_2$ is the only connected Dynkin diagram, so it is not 
$W(E_6)$-associate to the other two. Pick 
$w_1 \in \langle s_{\alpha_1},s_{\alpha_3} \rangle \cong S_3$ with
$w_1 (\alpha_1) = \alpha_3$. Then $(\gamma (w_1))(\alpha_6) = \alpha_5$ and
$w_1 \gamma (w_1) \in W(E_6)^{\Gamma_K}$. We conclude the $W(E_6)$-association classes on\\
$\{I_{2,1}, I_{2,2}, I_{2,3} \}$ are exactly the $W(E_6)^{\Gamma_K}$-association classes.

In the case $|I| = 3$, the possibilities are 
\[
I_{3,1} = \{ \alpha_3, \alpha_4, \alpha_5\},\;
I_{3,2} = \{ \alpha_2, \alpha_3, \alpha_5\},\; 
I_{3,3} = \{ \alpha_1, \alpha_2, \alpha_6\},\; 
I_{3,4} = \{ \alpha_1, \alpha_4, \alpha_6\}. 
\]
Among these $I_{3,1} \cong A_3$ is the only connected diagram, so it is not $W(E_6)$-associate 
to the other three. The sets $I_{3,2}$ and $I_{3,3}$ are associated via $w_1 \gamma (w_1)$, 
while the sets $I_{3,3}$ and $I_{3,4}$ are associated via $w_2$ (as above). Hence $\{ I_{3,2},
I_{3,3}, I_{3,4}\}$ forms one $W(E_6)^{\Gamma_K}$-association class and one
$W(E_6)$-association class.

If $I$ has rank 4, it is one of
\begin{align*}
& \{\alpha_1,\alpha_3,\alpha_5,\alpha_6\} \cong A_2 \times A_2 ,\\
& \{\alpha_1,\alpha_2,\alpha_4,\alpha_6\} \cong A_2 \times A_1 \times A_1,\\
& \{\alpha_2,\alpha_3,\alpha_4,\alpha_5\} \cong A_4 .
\end{align*}
These three are mutually non-isomorphic, so they form three association classes, both
for $W(E_6)$ and for $W(E_6)^{\Gamma_K}$.

When $|I| = 5$, we have the options
\[
E_6 \setminus \{\alpha_2\} \cong A_5 \quad \text{and} \quad
E_6 \setminus \{\alpha_4\} \cong A_2 \times A_2 \times A_1 .
\]
These are not isomorphic, so they form two association classes, both
for $W(E_6)$ and for $W(E_6)^{\Gamma_K}$.
\end{proof}
\vspace{2mm}

\section{Connected linear algebraic groups}
\label{sec:2}

The previous results about reductive groups can be generalized to all linear algebraic groups.
This relies mainly on t{he theory initiated by Borel and Tits \cite{BoTi}, and worked out
much further by Conrad, Gabber and Prasad \cite{CGP,CP}.

Let $\mc G$ be a connected linear algebraic $K$-group. We recall from \cite[Theorem 4.3.7]{Spr}
that $\mc G$ is irreducible and smooth as $K$-variety. In particular it is a smooth affine
group -- the terminology used in \cite{CGP}.

When $\mc G$ has a Levi decomposition, it is clear how Levi subgroups of $\mc G$ can be defined: 
as a Levi subgroup (in the sense of the previous section) of a Levi factor of $\mc G$. However, 
there exist linear algebraic groups that do not admit any Levi decomposition, even over 
$\overline K$ \cite[Appendix A.6]{CGP}. For those we do not know a good notion of Levi subgroups.

Instead we investigate a closely related kind of subgroups, already present in \cite{Spr}. 
Fix a $K$-rational cocharacter $\lambda : GL_1 \to \mc G$ and put
\begin{align*}
& \mc P_{\mc G}(\lambda) = \{ g \in \mc G : \lim_{a \to 0} \lambda (a) g \lambda (a)^{-1}
\text{ exists in } \mc G \} , \\
& \mc U_{\mc G}(\lambda) = \{ g \in \mc G : \lim_{a \to 0} \lambda (a) g \lambda (a)^{-1} = 1 \} ,\\
& \mc Z_{\mc G}(\lambda) = \mc P_{\mc G}(\lambda) \cap \mc P_{\mc G}(\lambda^{-1})
\end{align*}
These are $K$-subgroups of $\mc G$ \cite[Lemma 2.1.5]{CGP}. Moreover $\mc U_{\mc G}(\lambda)$ 
is $K$-split unipotent \cite[Proposition 2.1.10]{CGP}, and there 
is a Levi-like decomposition \cite[Proposition 2.1.8]{CGP}
\begin{equation}\label{eq:2.1}
\mc P_{\mc G}(\lambda) = \mc Z_{\mc G}(\lambda) \ltimes \mc U_{\mc G}(\lambda) . 
\end{equation}
By \cite[Lemma 2.1.5]{CGP} $Z_{\mc G}(\lambda)$ is the (scheme-theoretic) centralizer of  
$\lambda (GL_1)$, a $K$-split torus in $\mc G$. More generally, if $\mc S'$ is any $K$-split
torus in $\mc G$, $Z_{\mc G}(\mc S')$ is of the form $Z_{\mc G}(\lambda)$. Namely, for a
$K$-rational cocharacter $\lambda : GL_1 \to \mc S'$ whose image does not lie in the kernel 
of any of the roots of $(\mc G,\mc S')$.

Let $\mc R_{u,K}(\mc G)$ denote the unipotent $K$-radical of $\mc G$. By definition,
a pseudo-parabolic $K$-subgroup of $\mc G$ is a group of the form
\[
\mc P_\lambda := \mc P_{\mc G}(\lambda) \mc R_{u,K}(\mc G)
\quad \text{for some $K$-rational cocharacter } \lambda : GL_1 \to \mc G .
\]
Similarly we define
\[
\mc L_\lambda := \mc P_\lambda \cap \mc P_{\lambda^{-1}} = Z_{\mc G}(\lambda) \mc R_{u,K}(\mc G) . 
\]
We call $\mc L_\lambda$ a pseudo-Levi subgroup of $\mc G$. Just a like a Levi subgroup of a
reductive group is intersection of a parabolic subgroup with an opposite parabolic, a pseudo-Levi
subgroup is the interesection of a pseudo-parabolic subgroup with an opposite pseudo-parabolic.
We note that $\mc L_\lambda$ contains the centralizer of the $K$-split torus $\lambda (GL_1)$, 
but it may be strictly larger than the latter.

Unfortunately the groups $\mc P_\lambda$ and $\mc L_\lambda$ do in general not fit in a 
decomposition like \eqref{eq:2.1}, because $\mc U_{\mc G}(\lambda)$ may intersect 
$\mc R_{u,K}(\mc G)$ nontrivially. When $\mc G$ is pseudo-reductive over $K$ (that is,
$\mc R_{u,K}(\mc G) = 1$), the groups $\mc P_\lambda$ and $\mc L_\lambda$ coincide with
$\mc P_{\mc G}(\lambda)$ and $Z_{\mc G}(\lambda)$, respectively. In view of the remarks after
\eqref{eq:2.1}, the pseudo-Levi $K$-subgroups of a pseudo-reductive group are precisely the
centralizers of the $K$-split tori in that group.

More specifically, when $\mc G$ is reductive, the $\mc P_\lambda$ are precisely the parabolic 
subgroups of $\mc G$ \cite[Proposition 2.2.9]{CGP}, the $\mc L_\lambda$ are the Levi subgroups 
of $\mc G$ and \eqref{eq:2.1} is an actual Levi 
decomposition of $\mc P_\lambda$. This justifies our terminology ``pseudo-Levi subgroup''.

The notions pseudo-parabolic and pseudo-Levi are preserved under separable extensions
of the base field $K$ \cite[Proposition 1.1.9]{CGP}, but not necessarily under inseparable
base-change. This is caused by the corresponding behaviour of the unipotent $K$-radical. 

We consider the $K$-group $\mc G' := \mc G / \mc R_{u,K}(\mc G)$, the maximal pseudo-reductive
quotient of $\mc G$.

\begin{lem}\label{lem:2.1}
There is a natural bijection between the sets of pseudo-parabolic $K$-subgroups of $\mc G$
and of $\mc G'$. It remains a bijection if we take $K$-rational conjugacy classes on both
sides.
\end{lem}
\begin{proof}
The map sends $\mc P_\lambda$ to $\mc P'_\lambda := \mc P_\lambda / \mc R_{u,K}(\mc G)$.
It is bijective by \cite[Proposition 2.2.10]{CGP}. According to \cite[Proposition 3.5.7]{CGP}
every pseudo-parabolic subgroup of $\mc G$ (or of $\mc G'$) is its own scheme-theoretic
normalizer. Hence the variety of $\mc G (K)$-conjugates of $\mc P_\lambda$ is
$\mc G (K) / \mc P_\lambda (K)$. By \cite[Lemma C.2.1]{CGP} this is isomorphic with
$(\mc G / \mc P_\lambda ) (K)$. Next \cite[Proposition 2.2.10]{CGP} tells us that the 
$K$-varieties $\mc G / \mc P_\lambda$ and $\mc G' / \mc P'_\lambda$ can be identified. We obtain
\[
\mc G (K) / \mc P_\lambda (K) \cong (\mc G / \mc P_\lambda ) (K) \cong 
(\mc G' / \mc P_\lambda' ) (K) \cong \mc G' (K) / \mc P'_\lambda (K) ,
\]
where the right hand side can be interpreted as the variety of $\mc G' (K)$-conjugates of
$\mc P'_\lambda$. It follows that two pseudo-parabolic $K$-subgroups $\mc P_\lambda$ and
$\mc P_\mu$ are $\mc G (K)$-conjugate if and only if  $\mc P'_\lambda$ and 
$\mc P'_\mu$ are $\mc G' (K)$-conjugate.
\end{proof}

The setup from the start of Section \ref{sec:1} (with $\mc S, \mc T, \Delta_0, 
\ldots$) remains valid for the current $\mc G$, when we reinterpret $\mc B$ as a minimal 
pseudo-parabolic $K_s$-subgroup of $\mc G$. (Also, the $K$-group $Z_{\mc G}(\mc S)$ is not 
always pseudo-Levi in $\mc G$, for that we still have to add $\mc R_{u,K}(\mc G)$ to it.)
We refer to \cite[Proposition C.2.10 and Theorem C.2.15]{CGP} for the proofs in this generality.

The set of simple roots $\Delta_K$ for $(\mc G,\mc S)$ can again be identified with
$(\Delta \setminus \Delta_0 ) / \mu_{\mc B}(\Gamma_K)$.
For every $\mu_{\mc B}(\Gamma_K)$-stable subset $I$ of $\Delta$ containing $\Delta_0$
we get a standard pseudo-parabolic $K$-subgroup $\mc P_I$ of $\mc G$. By Lemma \ref{lem:2.1} 
and \cite[Theorem 15.4.6]{Spr} every pseudo-parabolic $K$-subgroup is $\mc G (K)$-conjugate
to a unique such $\mc P_I$. The unicity implies that two pseudo-parabolic $K$-subgroups of $\mc G$ 
are $\mc G (K)$-conjugate if and only if they are $\mc G (K_s)$-conjugate. (Recall that by
\cite[Proposition 3.5.2.ii]{CGP} pseudo-parabolicity is preserved under base change from $K$ to 
$K_s$.) By \cite[Proposition 3.5.4]{CGP} (which can only be guaranteed when the fields are
separably closed, as pointed out to us by Gopal Prasad), $\mc G (K_s)$-conjugacy of 
pseudo-parabolic subgroups is equivalent to $\mc G (K)$-conjugacy.

Write $\mc P_I = \mc P_{\lambda_I}$ for some $K$-rational homomorphism 
$\lambda_I : GL_1 \to \mc S$. It is easy to see (from \cite[Lemma 15.4.4]{Spr} and Lemma
\ref{lem:2.1}) that $\mc P_{\lambda_I^{-1}}$ does not depend 
on the choice of $\lambda_I$, and we may denote it by $\mc P_{-I}$. Then we define
\[
\mc L_I := \mc P_I \cap \mc P_{-I} = 
\mc P_{\lambda_I} \cap \mc P_{\lambda_I^{-1}} = \mc L_{\lambda_I} . 
\]
We call $\mc L_I$ a standard pseudo-Levi subgroup of $\mc G$. It is the inverse image,
with respect to the quotient map $\mc G \to \mc G'$, of the (standard pseudo-Levi) 
$K$-subgroup of $\mc G'$ called $L_I$ in \cite[Lemma 15.4.5]{Spr}. In the introduction we
called this $\mc L_{I_K}$, which relates to $\mc L_I$ by $I_K = (I \setminus \Delta_0) 
/ \mu_{\mc B}(\Gamma_K)$.

We are ready to generalize Lemma \ref{lem:1}.

\begin{lem}\label{lem:2.2}
\enuma{
\item Every pseudo-Levi $K$-subgroup of $\mc G$ is $\mc G (K)$-conjugate to a standard Levi
$K$-subgroup of $\mc G$.
\item For two standard pseudo-Levi $K$-subgroups $\mc L_I$ and $\mc L_J$ the following are
equivalent:
\begin{enumerate}[(i)]
\item $\mc L_I$ and $\mc L_J$ are $\mc G (K)$-conjugate;
\item $(I \setminus \Delta_0) / \mu_{\mc B}(\Gamma_K)$ and 
$(J \setminus \Delta_0) / \mu_{\mc B}(\Gamma_K)$ are $W (\mc G, \mc S)$-associate.
\end{enumerate} }
\end{lem} 
\begin{proof}
(a) Let $\mc L_\lambda$ be a pseudo-Levi $K$-subgroup of $\mc G$. Because $\mc P_\lambda$
is $\mc G (K)$-conjugate to a standard pseudo-parabolic $K$-subgroup $\mc P_I$ of $\mc G$, 
we may assume that 
\[
\mc L_\lambda \subset \mc P_\lambda = \mc P_I. 
\]
Since all maximal $K$-split tori of $\mc P_I$ are $\mc P_I (K)$-conjugate 
\cite[Theorem C.2.3]{CGP}, we may  further assume that the image of $\lambda$ is 
contained in $\mc S$. By \cite[Corollary 2.2.5]{CGP} the $K$-split unipotent radical 
$\mc R_{us,K}(\mc P_I)$ equals both $\mc U_{\mc G}(\lambda_I) \mc R_{u,K}(\mc G)$ and 
$\mc U_{\mc G}(\lambda) \mc R_{u,K}(\mc G)$ By \cite[Lemma 15.4.4]{Spr} the Lie algebra of 
$\mc P_I / \mc R_{u,K}(\mc G)$
can be analysed in terms of the weights for the adjoint action Ad$(\lambda)$ of $GL_1$ on
the Lie algebra of $\mc G'$. Namely, $\mc P_I / \mc R_{u,K}(\mc G)$ corresponds to the sum
of the subspaces on which $GL_1$ acts by characters $a \mapsto a^n$ with $n \in \Z_{\geq 0}$.
The Lie algebra of the subgroup 
\[
\mc R_{us,K}(\mc P_I) \mc R_{u,K}(\mc G) / \mc R_{u,K}(\mc G)
\]
is the sum of the subspaces on which Ad$(\lambda)$ acts as $a \mapsto a^n$ with $n \in \Z_{>0}$.
From \eqref{eq:2.1} inside $\mc G'$ we deduce that the Lie algebra of 
$\mc L_I / \mc R_{u,K}(\mc G)$ is the direct sum of the Lie algebra of $Z_{\mc G'}(\mc S)$ and 
the root spaces for roots $\alpha$ with $\langle \alpha ,\, \lambda \rangle = 0$. This holds for 
both $\lambda$ and $\lambda_I$, from which we conclude that $\mc L_\lambda = \mc L_I$.\\
(b) This can shown just as Lemma \ref{lem:1}.b, using in particular that the natural
map $N_{\mc G}(\mc S)(K) \to W(\mc G,\mc S)$ is surjective \cite[Proposition C.2.10]{CGP}.
\end{proof}

\begin{lem}\label{lem:2.7}
There is a natural bijection between the sets of pseudo-Levi $K$-subgroups of $\mc G$ and of 
$\mc G'$. It remains a bijection if we take $K$-rational conjugacy classes on both sides.
\end{lem}
\begin{proof}
The map sends $\mc L_\lambda$ to $\mc L'_\lambda := \mc L_\lambda / \mc R_{u,K}(\mc G)$.
This map is bijective for the same reason as in with pseudo-parabolic subgroups: 
$\mc G$ and $\mc G'$ have essentially the same tori, see \cite[Proposition 2.2.10]{CGP}.

By \cite[Theorem C.2.15]{CGP} the $K$-groups $\mc G$ and $\mc G'$ have the same root system 
and the same Weyl group. Then Lemma \ref{lem:2.2}.b says that set of the conjugacy classes of
pseudo-Levi $K$-subgroups are parametrized by the same data for both groups. Hence the
map $\mc L_I = \mc L_{\lambda_I} \mapsto \mc L'_{\lambda_I} = \mc L'_I$ also induces a bijection
between these sets of conjugacy classes.
\end{proof}

In case $\mc G'$ is reductive, Lemmas \ref{lem:2.1} and \ref{lem:2.7} furnish bijections
\begin{equation}\label{eq:2.2}
\begin{array}{ccc}
\hspace{-3mm} \{ \text{parabolic $K$-subgroups of } \mc G' \} & \longleftrightarrow &
\{ \text{pseudo-parabolic $K$-subgroups of } \mc G \} \hspace{-5mm} \\
\mc P_\lambda / \mc R_{u,K}(\mc G) = P_{\mc G'}(\lambda) & \leftrightarrow & \mc P_\lambda \\
\\
\{ \text{Levi $K$-subgroups of } \mc G' \} & \longleftrightarrow &
\{ \text{pseudo-Levi $K$-subgroups of } \mc G \} \hspace{-5mm} \\
\mc L_\lambda / \mc R_{u,K}(\mc G) = Z_{\mc G'}(\lambda) & \leftrightarrow & \mc L_\lambda
\end{array} 
\end{equation}
which induce bijections between the $K$-rational conjugacy classes on both sides.

We will now start to work towards the main result of this section:

\begin{thm}\label{thm:2.7}
Let $\mc G$ be a connected linear algebraic $K$-group. Any two pseudo-Levi $K$-subgroups of 
$\mc G$ which are $\mc G (\overline K)$-conjugate are already $\mc G (K)$-conjugate.
\end{thm}
The main steps of our argument are:
\begin{itemize}
\item Reduction from the general case to absolutely pseudo-simple $K$-groups with trivial centre.
\item Proof when $\mc G$ quasi-split over $K$ (i.e. $\Delta_0$ is empty).
\item Proof for absolutely pseudo-simple $K$-groups with trivial centre 
(using the quasi-split case).
\end{itemize}

\begin{lem}\label{lem:2.6}
Suppose that Theorem \ref{thm:2.7} holds for all absolutely pseudo-simple groups with 
trivial centre. Then it holds for all connected linear algebraic groups.
\end{lem}
\begin{proof}
By Lemma \ref{lem:2.7} we may just as well consider the pseudo-reductive group $\mc G' =
\mc G / \mc R_{u,K}(\mc G)$.
The derived group $\mc D (\mc G')$ has the same root system and Weyl group as $\mc G'$, both
over $K$ and over $K_s$, by \cite[Proposition 1.2.6 or Theorem C.2.15]{CGP}. In view of 
Lemma \ref{lem:2.2}, we may replace $\mc G'$ by $\mc D (\mc G')$. In particular $\mc G'$ is now 
pseudo-semisimple \cite[Remark 11.2.3]{CGP}.
Since the centre of $\mc G'$ is contained in every pseudo-Levi subgroup, we may 
divide it out. Thus we may assume that $Z(\mc G') = 1$, while retaining pseudo-reductivity 
\cite[Proposition 4.1.3]{CP}. 

Let $\{ \mc G'_j \}_j$ be the finite collection of normal pseudo-simple $K$-subgroups of $\mc G'$,
as in \cite[Propostion 3.1.8]{CGP}. The root system and Weyl group of $\mc G' (K)$ decompose 
as products of these objects for the $\mc G'_j (K)$. Combining that with Lemma \ref{lem:2.2} 
we see that it suffices to prove the theorem for each of the $\mc G'_j$. 

To simplify the notation, we assume from now on that $\mc G$ is a pseudo-simple $K$-group.
Let $\{ \mc G_i \}_i$ be the finite collection of normal pseudo-simple $K_s$-subgroups of $\mc G$.
These subgroups generate $\mc G$ as $K_s$-group \cite[Lemma 3.1.5]{CGP} and $\Gamma_K$ permutes 
them transitively. 
This serves as a slightly weaker analogue of \eqref{eq:1.4}. Next we can argue exactly as in
the proof of Lemma \ref{lem:3}, only replacing some parts by their previously established 
"pseudo"-analogues. As a consequence, it suffices to prove the theorem for the 
absolutely pseudo-simple groups $\mc G_i$ (over the field $K_i = K_s^{\Gamma_i}$).
If necessary, we can still divide out the centre of $\mc G_i$, as observed above for $\mc G'$.
\end{proof}

Following \cite[\S C.2]{CP} we say that a connected linear algebraic group $\mc G$ is 
quasi-split (over $K$) if a minimal pseudo-parabolic $K$-subgroup of $\mc G$ is also minimal 
as pseudo-parabolic $K_s$-subgroup. In view of the classification of conjugacy classes of
pseudo-parabolic $K_s$-subgroups, this condition is equivalent to $\Delta_0 = \emptyset$.

\begin{prop}\label{prop:2.3}
Theorem \ref{thm:2.7} holds when $\mc G$ is quasi-split over $K$.
\end{prop}
\begin{proof}
In view of Lemma \ref{lem:2.7} we may assume that $\mc G$ is pseudo-reductive. 
Consider the reductive $\overline K$-group $\mc G^\red := \mc G / \mc R_u (\mc G)$. 
The image of $\mc T$ in $\mc G^\red$ is a maximal torus of $\mc G^\red$. It is isomorphic to 
$\mc T$ via the projection map, and we may identify it with $\mc T$. Thus $\mc G^\red$ has a 
reduced (integral) root system $\Phi (\mc G^\red,\mc T)$. 
The maximal $K$-torus $\mc T$ of $\mc G$ splits over $K_s$. In the terminology of
\cite[Definition 2.3.1]{CGP}, $\mc G$ is pseudo-split over $K_s$. This is somewhat weaker than
split -- the root system $\Phi (\mc G,\mc T)$ is integral but not necessarily
reduced. (It can only be non-reduced if $K$ has characteristic 2.)

By \cite[Proposition 2.3.10]{CGP} the quotient map $\mc G \to \mc G^\red$ induces a bijection 
between $\Phi (\mc G^\red,\mc T)$ and $\Phi (\mc G,\mc T)$, provided that the latter is reduced.
In general $\Phi (\mc G^\red,\mc T)$ can be identified with the system of non-multipliable roots
in $\Phi (\mc G,\mc T)$. In particular these two root systems have the same Weyl group, and there
is a $W(\mc G,\mc T)$-equivariant bijection
\[
\begin{array}{ccc}
\{ \text{parabolic subsystems of } \Phi (\mc G,\mc T) \} & \longrightarrow & 
\{ \text{parabolic subsystems of } \Phi (\mc G^\red,\mc T) \} \\
R & \mapsto & R \cap \Phi (\mc G^\red,\mc T)  
\end{array}.
\]
This induces a bijection between the sets of simple roots for these root systems, say
$I \longleftrightarrow I^\red$. We note that
\begin{equation}\label{eq:2.3}
I,J \text{ are } W(\mc G,\mc T) \text{-associate} \; \Longleftrightarrow \; 
I^\red ,J^\red \text{ are } W(\mc G^\red,\mc T) \text{-associate}.
\end{equation}
By Lemma \ref{lem:2.2}.a it suffices to prove the lemma for standard pseudo-Levi $K$-subgroups 
$\mc L_I ,\mc L_J$. We assume that $\mc L_I$ and $\mc L_J$ are $\mc G (\overline K)$-conjugate. 
Then the pseudo-parabolic $\overline K$-subgroups 
\[
\mc L^\red_I = \mc L_I \mc R_u (\mc G) / \mc R_u (\mc G) \quad \text{and} \quad
\mc L^\red_J = \mc L_J \mc R_u (\mc G) / \mc R_u (\mc G)
\] 
of $\mc G^\red$ are conjugate. By Lemma \ref{lem:2.2}.b the associated sets of simple roots 
$I^\red$ and $J^\red$ are $W(\mc G^\red,\mc T)$-associate. Then \eqref{eq:2.3} and Lemma 
\ref{lem:2.2}.b entail that $\mc L_I$ and $\mc L_J$ are $\mc G (K_s)$-conjugate.

As $\mc G$ is quasi-split over $K$, the root system $\Phi (\mc G,\mc S)$ can be obtained by
a simple form of Galois descent: it consists of the $\Gamma_K$-orbits in $\Phi (\mc G,\mc T)$.
We know from \cite[Lemma 15.3.7]{Spr} that $W(\mc G,\mc S)$ is generated by the reflections
$s_\alpha$ with $\alpha \in \Phi (\mc G,\mc S)$. Let $\mc H$ be a quasi-split reductive $K$-group 
$\mc H$ with the same root datum as $\mc G^\red$, and the same $\Gamma_K$-action on that. 
By \cite[Proposition 2.4.2]{SiZi} (applied to $\mc H$), the aforementioned reflections generate 
the subgroup $W(\mc G,\mc T)^{\Gamma_K}$ of $W(\mc G,\mc T)$. Thus \eqref{eq:7} holds again.

We already showed that the $\Gamma_K$-stable subsets $I$ and $J$ of $\Delta$ are
$W(\mc G,\mc T)$-associate. By Lemma \ref{lem:1}.b the corresponding Levi $K$-subgroups
$\mc L_I^{\mc H} ,\mc L_J^{\mc H}$ of $\mc H$ are $\mc H (\overline K)$-conjugate. Then
Theorem \ref{thm:2} says that $\mc L_I^{\mc H}$ and $\mc L_J^{\mc H}$ are also 
$\mc H (K)$-conjugate. Again using Lemma \ref{lem:1}.b, we deduce that $I$ and $J$ are associate 
under $W(\mc G,\mc T)^{\Gamma_K} = W(\mc G,\mc S)$. 
Finally Lemma \ref{lem:2.2}.b tells us that $\mc L_I$ and $\mc L_J$ are $\mc G (K)$-conjugate.
\end{proof}

To go beyond quasi-split linear algebraic groups, we would like to use arguments like Lemmas 
\ref{lem:6} and \ref{lem:7}. However, the usual notion of an inner form (for reductive
groups) is not flexible enough for pseudo-reductive groups \cite[\S C]{CP}. Better results are
obtained by allowing inner twists involving a $K$-group of automorphisms called
$(\mr{Aut}^{sm}_{\mc D (\mc G) / K} )^\circ$ in \cite[\S C.2]{CP}. This leads to the notion of
pseudo-inner forms of pseudo-reductive groups. Every pseudo-reductive $K$-group admits a
quasi-split inner form, apart from some exceptions that can only occur if char$(K) = 2$ and
$[K : K^2] > 4$ \cite[Theorem C.2.10]{CP}.

\begin{lem}\label{lem:2.4}
Let $\mc G$ be a pseudo-reductive $K_s$-group and let $\lambda : GL_1 \to \mc G$ be a 
$K_s$-rational cocharacter. Suppose that $\phi \in (\mr{Aut}^{sm}_{\mc D (\mc G) / K_s} )^\circ (K_s)$
stabilizes the $K_s$-subgroups $\mc P_\lambda$ and $\mc L_\lambda$. 
Then $\phi$ stabilizes every $K_s$-subgroup of $\mc G$ that contains $\mc L_\lambda$.
\end{lem}
\begin{proof}
Since the centre of $\mc G$ is contained in $\mc L_\lambda$, we may divide it out. Thus we may
assume that $Z(\mc G) = 1$, while retaining pseudo-reductivity \cite[Proposition 4.1.3]{CP}. 
The derived group $\mc D (\mc G)$ is pseudo-semisimple \cite[Remark 11.2.3]{CGP} and 
$\mc G = Z_{\mc G}(\mc T) \mc D (G)$ \cite[Proposition 1.2.6]{CGP}. By \cite[Lemma 1.2.5.ii]{CGP}
the centre of $\mc D (\mc G)$ centralizes $\mc T$. Since $Z_{\mc G}(\mc T)$ is commutative
\cite[Proposition 1.2.4]{CGP}, $Z(\mc D (\mc G))$ commutes with it. Hence
\[
Z(\mc D (\mc G)) = Z( \mc G ) \cap \mc D (\mc G) = 1.
\]
We recall from \cite[\S 4.1.2]{CP} that $\mr{Aut}_{\mc G,Z_{\mc G}(\mc T)}$ is the $K_s$-group 
of automorphisms of $\mc G$ which restrict to the identity on the Cartan $K_s$-subgroup 
$Z_{\mc G}(\mc T)$. The maximal smooth closed $K_s$-subgroup $Z_{\mc G,Z_{\mc G}(\mc T)}$ of
$\mr{Aut}_{\mc G,Z_{\mc G}(\mc T)}$ is connected \cite[Proposition 6.1.4]{CP}. The same holds
with $\mc D (\mc G)$ and $\mc C := Z_{\mc D (\mc G)}(\mc T)$ instead of $\mc G$ and 
$Z_{\mc G}(\mc T)$. In fact, by \cite[Proposition 6.1.7]{CP} there is a natural isomorphism
\begin{equation}\label{eq:2.6}
Z_{\mc G,Z_{\mc G}(\mc T)} \to Z_{\mc D (\mc G),\mc C} .
\end{equation}
Embedding $\mc C$ diagonally in $\mc D(\mc G) \rtimes Z_{\mc D (\mc G),\mc C}$, one forms 
the $K_s$-group $(\mc D (\mc G) \rtimes Z_{\mc D(\mc G),\mc C}) / \mc C$. It naturally acts 
on $\mc G$, the part $\mc D (\mc G)$ by conjugation and $Z_{\mc D (\mc G),\mc C}$ via 
\eqref{eq:2.6}. According to \cite[Proposition 6.2.4]{CP}, which we may apply because 
$\mc D (\mc G)$ is pseudo-semisimple, there is an isomorphism of $K_s$-groups
\[
(\mc D (\mc G) \rtimes Z_{\mc D (\mc G),\mc C}) / \mc C \to 
(\mr{Aut}^{sm}_{\mc D (\mc G) / K_s} )^\circ ,
\]
which preserves the actions on $\mc G$. Furthermore \cite[Proposition 6.2.4]{CP} also says
that the homomorphism
\begin{equation}\label{eq:2.4}
\mc D (\mc G) (K_s) \rtimes Z_{\mc D (\mc G),\mc C}(K_s) \to  
(\mr{Aut}^{sm}_{\mc D (\mc G) / K_s} )^\circ (K_s)
\end{equation}
is surjective. By \cite[Proposition 6.1.4]{CP} there is a decomposition
\[
Z_{\mc D (\mc G),\mc C} \cong \prod\nolimits_{\alpha \in \Delta} Z_{\mc G_\alpha,\mc C_\alpha}
\qquad \text{as } K_s\text{-groups}.
\]
Taking into account that $Z(\mc D (\mc G)) = 1$, \cite[Lemma 6.1.3]{CP} and 
\cite[Proposition 9.8.15]{CGP} show that each of the $K_s$-groups $Z_{\mc G_\alpha,\mc C_\alpha}$ 
is a subtorus of $\mc T \cap \mc D (\mc G)$ which acts on $\mc G$ by conjugation. 
Combining that with \eqref{eq:2.4}, we deduce that $\phi$ can be realized as Ad$(g)$ for some 
$g \in \mc D (\mc G) (K_s)$.

As $\mc P_\lambda$ is its own normalizer \cite[Proposition 3.5.7]{CGP}, we must have $g \in
\mc P_\lambda (K_s)$. A nontrivial element $u$ of $\mc U_{\mc G}(\lambda)(K_s)$ cannot normalize 
$\mc L_\lambda$, because 
\[
\lambda (a) u \lambda^{-1}(a) u^{-1} \in \mc U_{\mc G}(\lambda)(K_s) \setminus \{1\}
\]
for generic (i.e. not a root of unity) $a \in K^\times$.  In view of \eqref{eq:2.1}, this 
implies that the normalizer of $\mc L_\lambda$ in $\mc P_\lambda$ is $\mc L_\lambda$ itself.
Thus the assumptions of the lemma even entail $g \in \mc L_\lambda (K_s)$. Now it is clear 
that $\phi = \mr{Ad}(g)$ stabilizes every $K_s$-subgroup of $\mc G$ that contains $\mc L_\lambda$.
\end{proof}

Suppose that $\mc G^*$ is a quasi-split pseudo-reductive group and that $\psi : \mc G \to \mc G^*$
is a pseudo-inner twist. (This forces $\mc G$ to be pseudo-reductive as well.) The setup 
leading to Lemma \ref{lem:6} remains valid if we replace all objects by their pseudo-versions.

\begin{lem}\label{lem:2.5}
Let $\mc H$ be a $K_s$-subgroup of $\mc G$ containing $\mc L_{\Delta_0}$. Then $\mc H$ is defined
over $K$ if and only if $\psi (\mc H)$ is defined over $K$.
\end{lem}
\begin{proof}
Exactly as in the proof of Lemma \ref{lem:6} one shows that $(\mc P^*_{\Delta_0},\mc L^*_{\Delta_0})$
is defined over $K$ and stable under Ad$(u(\gamma))$ for all $\gamma \in \Gamma_K$. Next
Lemma \ref{lem:2.4} says that Ad$(u(\gamma)) \in (\mr{Aut}^{sm}_{\mc D (\mc G) / K} )^\circ (K_s)$
stabilizes $\psi (\mc H)$. Then \eqref{eq:1.7} shows that $\psi (\mc H)$ is $\Gamma_K$-stable
if and only if $\mc H$ is $\Gamma_K$-stable.
\end{proof}

Now we can finish the proof of our main result.

\begin{prop}\label{prop:2.8}
Theorem \ref{thm:2.7} holds for absolutely pseudo-simple $K$-groups with trivial centre. 
\end{prop}
\begin{proof}
By Lemma \ref{lem:2.2}.a it suffices to consider two standard pseudo-Levi subgroups
$\mc L_I, \mc L_J$ which are $\mc G (\overline K)$-conjugate. As $\mc G$ becomes pseudo-split
over $K_s$, Proposition \ref{prop:2.3} tells us that there exists a $w \in \mc G (K_s)$ with
$w \mc L_I w^{-1} = \mc L_J$. 

By \cite[Proposition 4.1.3 and Theorem 9.2.1]{CP} $\mc G$ is generalized standard, in the
sense of \cite[Definition 9.1.7]{CP}. With \cite[Definition 9.1.5]{CP} we see that (at least) 
one of the following conditions holds:
\begin{enumerate}[(i)]
\item The characteristic of $K$ is not 2, or char$(K) = 2$ and $[K : K^2] \leq 4$.
\item The group $\mc G$ is standard \cite[Definition 2.1.3]{CP} or exotic 
\cite[Definitions 2.2.2 and 2.2.3]{CP}.
\item The root system of $\mc G$ over $K_s$ has type $B_n, C_n$ or $BC_n$ with $n \geq 1$.
\end{enumerate}
\textbf{(i) and (ii).} In the cases (i) and (ii) with $\mc G$ standard,
\cite[Theorem C.2.10]{CP} tells us that $\mc G$ has a quasi-split pseudo-inner form. If we are
in case (ii) with $\mc G$ non-standard and char$(K) = 2$, then $\mc G$ is an exotic 
pseudo-reductive group with 
root system (over $K_s$) of type $B_n ,C_n$ or $F_4$. By \cite[Proposition C.1.3]{CP} it has
a pseudo-split $K_s / K$-form. Since the Dynkin diagram of $\mc G$ admits no nontrivial
automorphisms, the group $\mr{Aut}^{sm}_{\mc G / K}$ is connected and every $K_s/K$-form of
$\mc G$ is pseudo-inner \cite[Proposition 6.3.4]{CP}. Thus, in the cases (i) and (ii) $\mc G$ 
has a quasi-split pseudo-inner form.

Now we argue as in the proof of Lemma \ref{lem:7}, using Lemma \ref{lem:2.5} instead of 
Lemma \ref{lem:6}.c. The hypothesis in Lemma \ref{lem:7} is fulfilled 
for quasi-split pseudo-reductive groups, by Proposition \ref{prop:2.3}. This shows that
$\mc L_J$ is $\mc G (K)$-conjugate to a pseudo-Levi factor of $\mc P_I$. In the proof of
Lemma \ref{lem:2.2}.a we checked that all such pseudo-Levi factors are $\mc P_I (K)$-conjugate,
so $\mc L_J$ is $\mc G (K)$-conjugate to $\mc L_I$.

\textbf{(iii).} The three types can be dealt with in the same way, so we only consider 
root systems $\Phi (\mc G,\mc T)$ of type $B_n$. Since this Dynkin diagram does not
admit any nontrivial automorphisms, the action $\mu_{\mc B}$ of $\Gamma_K$ is trivial.

Suppose first that $n \leq 2$. Then any two diferent subsets of $\Delta$ are not 
$W(\mc G,\mc T)$-associate, as is easily checked. Hence $I = J$ and $\mc L_I = \mc L_J$ 
in this case.

From now on we suppose that $\Phi (\mc G,\mc T)$ has type $B_n$ with $n > 2$.
We realize the root system of type $B_n$ in the standard 
way in $\Z^n$. Let $\alpha_1, \ldots, \alpha_{n-1}, \alpha_n$ be the vertices of $\Delta$, 
where $\alpha_i = e_i - e_{i+1}$ for $i < n$ and $\alpha_n = e_n$ is the short simple root.

By Lemma \ref{lem:2.2}.b there exists a $w \in W(\mc G,\mc T) = W(B_n)$ with $w I = J$.
When $I$ or $J$ equals $\Delta$, we immediately obtain $I = J$. Hence we may assume that
$I \subsetneq \Delta \supsetneq J$.
Let $m \in \Z_{\geq 0}$ be the smallest number such that $\alpha_{n-m} \notin I$. For $j < m$,
$\alpha_{n-j} \in I$ is the unique root in $\Delta$ which is connected to $\alpha_n$ by
a string of length $j$. As $\alpha_n$ is the unique short simple root and $w I \subset \Delta$,
it follows that $w (\alpha_{n-j}) = \alpha_{n-j}$ for all $j < m$. The same considerations
apply to $J$ and $w I = J$, so $\{ \alpha_{n+1-m},\ldots,\alpha_n \} \subset I \cap J$ is 
fixed pointwise by $w$ and $\alpha_{n-m} \neq I \cup J$. As 
\[
\mr{span}_\Z \{ \alpha_{n+1-m}, \ldots,\alpha_n \} = \mr{span}_\Z \{ e_n, e_{n-1},\ldots,e_{n+1-m} \},
\]
$w$ must lie in $W(B_{n-m})$.

Write $\Delta' = \{ \alpha_1, \ldots, \alpha_{n-1-m} \}$ and $\Delta'' = \{\alpha_{n+1-m}, 
\ldots, \alpha_n \}$, two orthogonal sets of simple roots. The standard pseudo-Levi 
$K$-subgroup $\mc L_{\Delta' \cup \Delta''}$ of $\mc G$ contains $\mc L_I$ and $\mc L_J$. 
Decomposing its root system in irreducible components gives
\[
\mc L_{\Delta' \cup \Delta''} = \mc L_{\Delta'} \mc L_{\Delta''} .
\]
The index of $\mc L_{\Delta'}$ consists of $\Delta', \Delta' \cap \Delta_0$ and the trivial 
action of $\Gamma_K$. Here $\Delta'$ has type $A_{n-1-m}$ and by \cite[Lemma 15.5.8]{Spr} 
the subset $\Delta'_0 := \Delta' \cap \Delta_0$ is stable under the nontrivial automorphism of 
$A_{n-1-m}$. As shown in \cite[\S 3.3.2]{Tit}, this implies that there exists a divisor
$d$ of $n-m$ such that
\[
\Delta' \setminus \Delta'_0 = \Z d \cap [1,\ldots,n-1-m] . 
\]
With \cite[\S 17.1]{Spr} we see that $(\Delta',\Delta'_0,\mr{triv})$ is the index of an inner 
form $\mc H$ of $GL_n$. Explicitly, we can take $\mc H (\Q) = GL_{(n-m)/d}(D)$ where $D$ is a 
division algebra whose centre equals the ground field $\Q$. As maximal $\Q$-split torus
$\mc S^{\mc H}(\Q)$ we take the diagonal matrices with entries in $\Q^\times$.

The isomorphism class of the
Dynkin diagram $I' := I \cap \Delta$ determines the isomorphism class of the standard Levi
$\Q$-subgroup $\mc L_{I'}^{\mc H}$ of $\mc H$. Namely, $\mc L_{I'}^{\mc H}(\Q)$ is a direct
product of groups $GL_{n_j}(D)$, where $\sum_j n_j = (n-m)/d$ and $I'$ has connected components
of sizes $d n_j - 1$.

The set of simple roots $J' := J \cap \Delta'$ is associate to $I'$ by $w \in W(B_{n-m})$,
so isomorphic to $I'$ as Dynkin diagram. It follows that the standard Levi $\Q$-subgroups 
$\mc L_{I'}^{\mc H}$ and $\mc L_{J'}^{\mc H}$ of $\mc H$ are isomorphic. That is, 
$\mc L_{J'}^{\mc H}$ is also a direct product of the groups $GL_{n_j}(D)$, but maybe situated
in a different (standard) position inside $GL_{(n-m)/d}(D)$. With a permutation $w'$ from 
$S_{(n-m)/d}$ we can bring them in the same position. Then $n \mc L_{I'}^{\mc H} n^{-1} =
\mc L_{J'}^{\mc H}$ for some $n \in N_{\mc H (\Q}(\mc S^{\mc H}(\Q))$ and
$w' I' = J'$, where $w'$ is the image of $n$ in $W(\mc H,\mc S^{\mc H}) \cong S_{(n-m)/d}$.

As $W(\mc H,\mc S^{\mc H}) = W(\mc L_{\Delta'},\mc S)$, we conclude that $I'$ and $J'$ are
associate by an element of $W(\mc L_{\Delta'},\mc S) \subset W(\mc G,\mc S)$. 
Since $\Delta'$ and $\Delta''$ are orthogonal, $w'$ fixes $I \cap \Delta'' = J \cap \Delta''$
pointwise. Hence $w' I = J$ and by Lemma \ref{lem:2.2}.b $\mc L_I$ and $\mc L_J$ are
$\mc G (K)$-conjugate.
\end{proof}

\end{document}